\documentclass[a4paper,10pt]{amsart}

\usepackage{blindtext}
\usepackage{quoting}
\usepackage{amsthm}
\usepackage[utf8]{inputenc}
\usepackage{verbatim}
\usepackage{amsmath}
\usepackage{amsfonts}
\usepackage{stackrel}
\usepackage{amssymb}

\usepackage[svgnames]{xcolor}
\definecolor{blue(munsell)}{rgb}{0.0, 0.5, 0.69}
\usepackage{graphicx}
\usepackage{lipsum}
\usepackage{svg}
\usepackage{bbm}
\usepackage{caption}
\usepackage{hyperref}
\usepackage{tcolorbox}
\usepackage{cleveref}
\usepackage{epigraph}
\usepackage{ebproof}
\usepackage[cal=boondoxo]{mathalfa}
\usepackage[all,cmtip]{xy}
\usepackage{mathtools}
\usepackage{color}
\usepackage{mathrsfs}
\usepackage{subfig}
\usepackage{framed}
\usepackage{tikz}
\usepackage{tikz-cd}
\usepackage{multirow}
\usepackage[percent]{overpic}

\hypersetup{
    colorlinks = true,
    linkbordercolor = {red},
    linkcolor ={teal},
    anchorcolor = {pink},
    citecolor =  {orange},
    filecolor = {teal},
    menucolor = {teal},
    runcolor =  {teal},
    urlcolor = {teal},
}

\usepackage[all,cmtip]{xy}

\usetikzlibrary{calc}
\usetikzlibrary{matrix,arrows}
\usepackage{tikz-cd}
\usepackage[square, sort, numbers]{natbib}
\usepackage{calligra} % for dedication

    {\endMakeFramed}

    {\endMakeFramed}

\makeatletter
\g@addto@macro\bfseries{\boldmath}
\makeatother
\theoremstyle{definition}
\newtheorem{thm}{Theorem}[subsection]
\newtheorem{exa}[thm]{Example}
\newtheorem{rem}[thm]{Remark}
\newtheorem{ach}[thm]{Achtung!}
\newtheorem*{structure*}{Structure}
\newtheorem*{thm*}{Theorem}
\newtheorem*{ach*}{Achtung!}

\newtheorem{con}[thm]{Construction}

\newtheorem{prop}[thm]{Proposition}
\newtheorem{defn}[thm]{Definition}
\newtheorem{cor}[thm]{Corollary}
\newtheorem{notation}[thm]{Notation}

\newcommand\Sat{\operatorname{Sat}}

\newcommand\lan{\mathsf{lan}}

\newcommand\Set{\operatorname{\bf Set}}

\newcommand\ca{\mathcal {A}}
\newcommand\cb{\mathcal {B}}
\newcommand\cc{\mathcal {C}}

\newcommand\cf{\mathcal {F}}

\newcommand\cg{\mathcal {G}}

\newcommand\ce{\mathcal {E}}

\newcommand\ck{\mathcal {K}}

\newcommand\cs{\mathcal {S}}
\newcommand\ct{\mathcal {T}}

\newcommand\cx{\mathcal {X}}
\newcommand\cy{\mathcal {Y}}

\DeclareFontFamily{U}{min}{}
\DeclareFontShape{U}{min}{m}{n}{<-> udmj30}{}

\title{Formal Model theory \& Higher Topology}
\author{Ivan Di Liberti$^\dag$}

\thanks{$^\dag$ This research was mostly developed during the PhD studies of the author and has been supported through the grant 19-00902S from the Grant Agency of the Czech Republic. The finalization of this research has been supported by the GACR project EXPRO 20-31529X and RVO: 67985840.}

\address{
\newline Ivan \textsc{Di Liberti}\newline
Institute of Mathematics\newline
Czech Academy of Sciences\newline
\v{Z}itn\'{a} 25, Prague, Czech Republic\newline
diliberti.math@gmail.com\newline
}

\begin{document}

\begin{abstract}
We study the $2$-categories BIon, of (generalized) bounded ionads, and $\text{Acc}_\omega$, of accessible categories with directed colimits, as an abstract framework to approach formal model theory. We relate them to topoi and (lex) geometric sketches, which serve as categorical specifications of geometric theories. We provide reconstruction and completeness-like results. We relate abstract elementary classes to locally decidable topoi. We introduce the notion of categories of saturated objects and relate it to atomic topoi.
\end{abstract}
\maketitle
\setcounter{tocdepth}{1}

{
  \hypersetup{linkcolor=black}
  \tableofcontents
}

\section*{Introduction}
\subsection*{General setting}
In the late 1980's Makkai and Paré presented their book \textit{Accessible categories: the foundations of categorical model theory} \cite{Makkaipare}, providing a solid framework that could accommodate a large portion of categorical logic. In the fashion of abstract logic and abstract model theory, the book has two main aspects: one semantical and one syntactic. On the one hand they introduced the theory of accessible categories\footnote{Which had already appeared under a different name in the work of Lair and Rosický.}, these are abstract categories of models of some theory. On the other hand they present the theory of \textit{sketches}\footnote{Which had been developed by the French school.}, which provide a categorical specification of infinitary first order theories. The interplay between sketches (syntax) and accessible categories (semantics) is a large portion of categorical model theory.

 Since then, categorical model theory has evolved significantly, thanks to the contribution of several authors, including the authors of the above mentioned book. The study of accessible categories from the point of view of the model theorist has led to the individuation of special classes of accessible categories, that best suit the most natural constructions of model theory. Among the most common additional requirements, we find:

\begin{itemize}
  \item the existence of directed colimits;
  \item amalgamation property (AP);
  \item joint embedding property (JEP);
  \item every morphism is a monomorphism;
  \item the existence of a (very) well behaved faithful functor $\ca \to \Set$ preserving directed colimits.
  \end{itemize}

Each of these different assumptions is motivated by some model theoretic intuition. For example, the request that every morphism is a monomorphism is motivated by the focus on elementary embedding, rather than homomorphisms of structures. The faithful functor into $\Set$ allows to construct directed colimits of models as colimits of underlying structures. The combination of (AP), (JEP) and the existence of directed colimits allows the construction of saturated objects \cite{Rsaturated}.

 Synthetizing the conjoint work of Beke, Rosický, Lieberman, Vasey et al. (see for example \cite{aec,internalsize,lieberman2015limits,internalsize,LB2014,universal,vasey2019accessible}) in a sentence, accessible categories with directed colimits generalize Shelah's framework of abstract elementary classes, and are special enough to recover the main features of categorical model theory. 

\subsection*{Our contribution}
In this paper we shift the focus from categorical model theory to \textit{formal model theory}. By this we intend that instead of studying the property of a category of models via its objects and arrows, we study it via its relational behavior among categories of models of theories. We introduce two specific incarnations of formal model theory.

\begin{itemize}
    \item[$\star$] the $2$-category $\text{Acc}_\omega$, of \textbf{accessible categories with directed colimits}, where $1$-cells are functors preserving directed colimits and $2$-cells are natural transformations.  
\end{itemize}

The study of this $2$-category is very coherent with the classical tradition à la Makkai-Paré, and the additional assumptions that we have listed above, will re-emerge in this setting, depending on the kind of constructions and behavior typical of model theory that we want to simulate.

\begin{itemize}
   \item[$\star$] the $2$-category $\text{BIon}$, of (generalized) bounded ionads.
\end{itemize}

The first notion of ionads was introduced by Garner \cite{ionads}, mainly from a topological point of view. In this paper we introduce the notion of \textbf{ionad of models} of a geometric theory, and we give a ionadic interpretation of Makkai's Ultracategories. 

On the syntactic side of this paper we find \textbf{topoi} and (lex) \textbf{geometric sketches}, these are both categorical specifications of geometric theories, as we discuss in the background section.
 
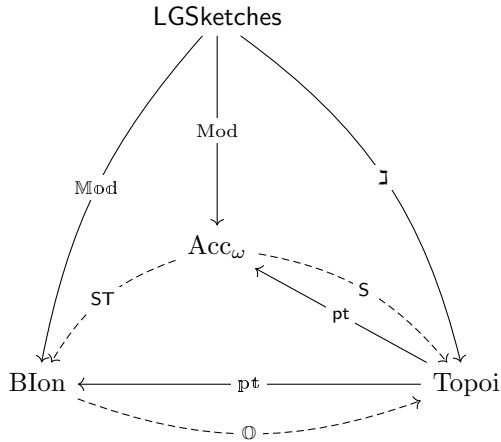
\begin{center}
\begin{tikzcd}
                                                          & \mathsf{LGSketches} \arrow[rrdddddd, "\gimel" description, bend left=20] \arrow[dddd, "\text{Mod}" description,] \arrow[ldddddd, "\mathbb{M}\mathbbm{od}" description, bend right=15] &  &                                                                                                       \\
                                                          &                                                                                                                                                                                  &  &                                                                                                       \\
                                                          &                                                                                                                                                                                  &  &                                                                                                       \\
                                                          &                                                                                                                                                                                  &  &                                                                                                       \\
                                                          & \text{Acc}_\omega \arrow[ldd, "\mathsf{ST}" description, dashed, bend right=20] \arrow[rrdd, "\mathsf{S}" description, dashed, bend left=20]                                                             &  &                                                                                                       \\
                                                          &                                                                                                                                                                                  &  &                                                                                                       \\
\text{BIon} \arrow[rrr, "\mathbb{O}" description, dashed, bend right=20] &                                                                                                                                                                                  &  & \text{Topoi} \arrow[lluu, "\mathsf{pt}" description,] \arrow[lll, "\mathbbm{pt}" description,]
\end{tikzcd}
\end{center}

In this paper we study the interplay between theories (topoi and sketches) and categories of models (ionads and accessible categories), providing reconstruction results for both of them. This will amount to a complete description of the diagram above.

From a technical point of view we build on two previous paper of ours \cite{thcat,thgeo}, where we develop relationships between topoi, bounded ionads, and accessible categories with directed colimits.

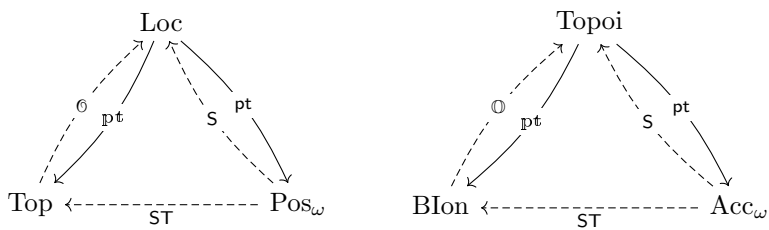
\begin{center}
\begin{tikzcd}
                                                                      & \text{Loc} \arrow[lddd, "\mathbbm{pt}" description, bend left=12] \arrow[rddd, "\mathsf{pt}" description, bend left=12] &                                                                                                                 \\
                                                                      &                                                                                                                  &                                                                                                                 \\
                                                                      &                                                                                                                  &                                                                                                                 \\
\text{Top} \arrow[ruuu, "\mathcal{O}" description, dashed, bend left=12] &                                                                                                                  & \text{Pos}_{\omega} \arrow[luuu, "\mathsf{S}" description, dashed, bend left=12] \arrow[ll, "\mathsf{ST}", dashed]
\end{tikzcd}
 \qquad
\begin{tikzcd}
                                                                      & \text{Topoi} \arrow[lddd, "\mathbbm{pt}" description, bend left=12] \arrow[rddd, "\mathsf{pt}" description, bend left=12] &                                                                                                                 \\
                                                                      &                                                                                                                  &                                                                                                                 \\
                                                                      &                                                                                                                  &                                                                                                                 \\
\text{BIon} \arrow[ruuu, "\mathbb{O}" description, dashed, bend left=12] &                                                                                                                  & \text{Acc}_{\omega} \arrow[luuu, "\mathsf{S}" description, dashed, bend left=12] \arrow[ll, "\mathsf{ST}", dashed]
\end{tikzcd}
\end{center}

In \cite{thgeo} we have categorified the Scott topology on a poset with directed joins and the Isbell duality between locales and topological spaces. We will briefly recall the results on those papers in the first sections, contextualizing them in the framework of Lawvere functorial semantics.

\subsection*{Structure} 

The exposition is organized as follows:
\begin{enumerate}
  \item[\S \ref{back}] We put together the needed background on accessible categories, ionads, topoi and sketches. The section can be completely skipped by the reader that is well versed with this topic and is intended to be a soft interaction for those readers whose background is closer to classical model theory. We provide several references and we contextualize most of the definition.
  \item[\S \ref{logicgeneralizedaxiom}] We recall the most relevant results of \cite{thgeo}, the \textbf{Scott adjunction} and the \textbf{categorified Isbell duality}, putting them in the context of functorial semantics.  We traces back the Scott topos to the seminal works of Linton and Lawvere on algebraic theories and algebraic varieties.
  \item[\S \ref{logicclassifyingtopoi}] This section is dedicated to reconstruction results. First we answer the question \textit{is there any relation between Scott topoi} and \textit{classifying topoi}? with a partially affirmative answer. Indeed every theory $\cs$ has a category of models $\mathsf{Mod}(\cs)$, but this category does not retain enough information to recover the theory, even when the theory has enough points (\Cref{classificatore}). That's why the Scott adjunction is not sharp enough. Nevertheless, every theory has a ionad of models $\mathbb{M}\mathbbm{od}(\cs)$, the category of opens of such a ionad $\mathbb{O}\mathbb{M}\mathbbm{od}(\cs)$ recovers theories with enough points (\Cref{isbellclassificatore}).
  \item[\S \ref{logicaec}] This section describes the relation between the Scott adjunction and abstract elementary classes, providing a restriction of the Scott adjunction to one between accessible categories where every map is a monomorphism and locally decidable topoi (\Cref{LDCAECs}).
  \item[\S \ref{logicsaturatedobjects}] In this section we give the definition of \textit{category of saturated objects} (CSO) and show that the Scott adjunction restricts to an adjunction between CSO and atomic topoi (\Cref{thmcategoriesofsaturated objects}). This section can be understood as an attempt to conceptualize the main result in \cite{simon}.
  \end{enumerate}

\subsection*{Notations and conventions} \label{backgroundnotations}

Most of the notation will be introduced when needed and we will try to make it as natural and intuitive as possible, but we would like to settle some notation.
\begin{enumerate}
\item $\ca, \cb$ will always be accessible categories, very possibly with directed colimits.
\item $\cx, \cy$ will always be ionads.
\item $\mathsf{Ind}_\lambda$ is the free completion under $\lambda$-directed colimits.
\item $\ca_{\kappa}$ is the full subcategory of $\kappa$-presentable objects of $\ca$.
\item $\cg, \ct, \cf, \ce$ will be Grothendieck topoi.
\item In general, $C$ is used to indicate small categories.
\item $\eta$ is the unit of the Scott adjunction.
\item $\epsilon$ is the counit of the Scott adjunction.
\item  $\P(X)$ is the category of small copresheaves of $X$.
\item An Isbell topos is a topos of the form $\mathbb{O}(\cx)$, for some bounded ionad $\cx$;
\item A Scott topos is a topos of the form $\mathsf{S}(\ca)$ for some accessible category $\ca$ with directed colimits.
\end{enumerate}

\begin{notation}[Presentation of a topos]\label{presentation of topos}
A presentation of a topos $\cg$ is the data of a  geometric embedding into a presheaf topos $f^*: \Set^{C} \leftrightarrows \cg :  f_*$. This means precisely that there is a suitable topology $\tau_f$ on $C$ that turns $\cg$ into the category of sheaves over $\tau$; in this sense $f$ \textit{presents} the topos as the category of sheaves over the site $(C, \tau_f)$. 
\end{notation}

\section{Background: accessibility, sketches, topoi and ionads} \label{back}
This section is merely expository and can be skipped by the reader that is well versed with the topics that it introduces.

\subsection{Accessible categories} \label{backgroundLPAC}

The theory of accessible and locally presentable categories has gained quite some popularity along the years because of its natural ubiquity. Most of the categories of the \textit{working mathematician} are accessible, with a few (but still extremely important) exceptions. For example, the category $\mathsf{Top}$ of topological spaces is not accessible. In general, categories of algebraic structures are locally  $\aleph_0$-presentable and many relevant categories of geometric nature are $\aleph_1$-accessible. A sound rule of thumb is that locally finitely presentable categories correspond to categories of models essentially algebraic theories, in fact this is even a theorem in a proper sense \cite[Chap. 3]{adamekrosicky94}. A similar intuition is available for accessible categories too, but some technical price must be paid \cite[Chap. 5]{adamekrosicky94}. Accessible and locally presentable categories (especially the latter) are \textit{tame} enough to make many categorical wishes come true; that's the case for example of the adjoint functor theorem, that has a very easy to check version for locally presentable categories.

\begin{ach} In this section $\lambda$ is a regular cardinal.
\end{ach}

\begin{defn}[$\lambda$-accessible category]
A $\lambda$-accessible category $\ca$ is a category with $\lambda$-directed colimits with a set of $\lambda$-presentable objects that generate by $\lambda$-directed colimits. An accessible category is a category that is $\lambda$-accessible for some $\lambda$.
\end{defn}

\begin{defn}[Locally $\lambda$-presentable category]
A locally $\lambda$-presentable category is a cocomplete $\lambda$-accessible category. A locally presentable category is a category that is locally $\lambda$-presentable for some $\lambda$.
\end{defn}

\begin{defn}[$\lambda$-presentable object]
An object $a \in \ca$ is $\lambda$-presentable if its covariant hom-functor $\ca(a, -): \ca \to \Set$ preserves $\lambda$-directed colimits.
\end{defn}

\begin{defn}[$\lambda$-directed posets and $\lambda$-directed colimits]
A poset $P$ is $\lambda$-directed if it is non empty and for every $\lambda$-small\footnote{This means that its cardinality is strictly less then $\lambda$. For example $\aleph_0$-small means finite.} family of elements $\{p_i\} \subset P$, there exists an upper bound. A $\lambda$-directed colimit is the colimit of a diagram over a $\lambda$-directed poset (seen as a category).
\end{defn}

\begin{notation} For a category $\ca$, we will call $\ca_\lambda$ its full subcategory of $\lambda$-presentable objects.
\end{notation}

\subsubsection{Literature}

There are two main references for the theory of accessible and locally presentable categories, namely \cite{adamekrosicky94} and \cite{Makkaipare}. The first one is intended for a broader audience and appeared few years after the second one. The second one is mainly concerned with the logical aspects of this theory. Another good general exposition is  \cite[Chap. 5]{BOR2}.

\subsubsection{Accessible categories and (infinitary) logic}
Accessible categories have been connected to (infinitary) logic in several (partially independent) ways. This story is recounted in Chapter 5 of \cite{adamekrosicky94}. Let us recall two of the most important results of that chapter.
\begin{enumerate}
  \item As locally presentable categories, accessible categories are categories of models of theories, namely \textit{basic} theories \cite[Def. 5.31, Thm. 5.35]{adamekrosicky94}.
  \item Given a theory $T$ in $L_\lambda$ the category $\mathsf{Elem}_\lambda(T)$ of models and $\lambda$-elementary embeddings is accessible \cite[Thm. 5.42]{adamekrosicky94}.
\end{enumerate} 
Unfortunately, it is not true in general that the whole category of models and homomorphisms of a theory in $L_\lambda$ is accessible. It was later shown by Lieberman \cite{Lthesis} and independently by Rosický and Beke \cite{aec}  that abstract elementary classes are accessible too. The reader that is interested in this connection might find interesting \cite{vasey2019accessible}, whose language is probably the closest to that of a model theorist.

\subsection{Sketches} \label{backgroundsketches}

\begin{defn}[Sketch]
A sketch is a quadruple $\mathcal{S}=(S, L,C, \sigma)$ where
\begin{enumerate}
  \item[$S$] is a small category;
  \item[$L$] is a class of diagrams in $S$, called \textit{limit} diagrams;
  \item[$C$] is a class of diagrams in $S$, called \textit{colimit} diagrams;
  \item[$\sigma$] is a function assigning to each diagram in $L$ a cone and to each diagram in $C$ a cocone.
\end{enumerate}
\end{defn}

\begin{defn}
A sketch is
\begin{itemize}

  \item \textit{limit} if $C$ is empty;
  \item \textit{colimit} if $L$ is empty;
  \item \textit{mixed} (used only in  emphatic sense) if it's not limit, nor colimit;
  \item \textit{geometric}  if each cone is finite;
  \item \textit{coherent} if it is geometric and and every cocone is either finite or discrete, or it is a regular-epi specification\footnote{See \cite[D2.1.2]{elephant2}.}.
  
\end{itemize}
\end{defn}

\begin{defn}[Morphism of Sketches]
Given two sketches $\cs$ and $\ct$, a morphism of sketches $f: \cs \to \ct$ is a functor $f: S \to T$ mapping (co)limit diagrams into (co)limits diagrams and proper (co)cones into (co)cones.
\end{defn}

\begin{defn}[$2$-category of Sketches]
The $2$-category of sketches has sketches as objects, morphism of sketches as $1$-cells and natural transformations as $2$-cells.
\end{defn}

\begin{defn}[Category of models of a sketch]
For a sketch $\cs$ and a (bicomplete) category $\cc$, the category $\mathsf{Mod}_{\cc}(\cs)$ of $\cc$-models of the sketch is the full subcategory of $\cc^\cs$ of those functors that are models. If it's not specified, by $\mathsf{Mod}(\cs)$ we mean $\mathsf{Mod}_{\Set}(\cs)$.
\end{defn}

\begin{defn}[Model of a sketch]
A model of a sketch $\cs$ in a category $\cc$ is a functor $f: \cs \to \cc$ mapping each specified (co)cone to a (co)limit (co)cone.  If it's not specified a model is a $\Set$-model.
\end{defn}

\subsubsection{Literature}
There exists a plethora of different and yet completely equivalent approaches to the theory of sketches. We stick to the one that suits best our setting, following mainly \cite[Chap. 5.6]{BOR2} or \cite[Chap. 2.F]{adamekrosicky94}. Other authors, such as \cite{Makkaipare} and \cite{elephant2} use a different (and more classical) definition involving graphs. Sketches are normally used as generalized notion of theory. From this perspective these approaches are completely equivalent, because the underlying categories of models are the same. \cite[page 40]{Makkaipare} stresses that the graph-definition is a bit more flexible in \textit{daily practice}. Sketches were introduced by C. Ehresmann. Guitart, Lair and Burroni should definitely be mentioned among the relevant contributors. This list of references does not do justice to the French school, which has been extremely prolific on this topic, yet, for the purpose of this paper the literature above will be more then sufficient.

\subsubsection{Sketches: logic and sketchable categories}
Sketches became quite common among category theorists because of their expressiveness. In fact, they can be used as a categorical analog of those theories that can be axiomatized by (co)limit properties. For example, in the previous section, essentially algebraic theories are precisely those axiomatizable by finite limits.

\subsubsection{From theories to sketches}
We have mentioned that a sketch can be seen as a kind of theory. This is much more than a motto, or a motivational presentation of sketches. In fact, given a (infinitary) first order theory $\mathbb{T}$, one can always construct in a more or less canonical way a sketch $\mathcal S_{\mathbb{T}}$ whose models are precisely the models of $\mathbb{T}$. This is very well explained in \cite[D2.2]{elephant2}; for the sake of exemplification, let us state the theorem which is most relevant to our context.

\begin{thm}
If $\mathbb{T}$ is a (geometric) (coherent) theory, there there exists a (geometric) (coherent) sketch having the same category of models of $\mathbb{T}$.
\end{thm}

Some readers might be unfamiliar with geometric and coherent theories; these are just very specific fragments of first order (infinitary) logic. For a very detailed and clean treatment we suggest \cite[D1.1]{elephant2}. Sketches are quite a handy notion of theory because we can use morphisms of sketches as a notion of translation between theories. 

\begin{prop}[{\cite[Ex. 5.7.14]{BOR2}}]
If $f: \mathcal{S} \to\mathcal{T} $ is a morphism of sketches, then composition with $f$ yields an (accessible) functor $\mathsf{Mod}(\cs) \to \mathsf{Mod}(\ct)$.
\end{prop}

\subsubsection{Sketchability}
It should not be surprising that sketches can be used to \textit{axiomatize} accessible and locally presentable categories too. The two following results appear, for example, in \cite[2.F]{adamekrosicky94}.

% \begin{thm}
% A category is locally presentable if and only if it's equivalent to the category of models of a limit sketch.
% \end{thm}

\begin{thm}
A category is accessible if and only if it's equivalent to the category of models of a mixed sketch.
\end{thm}

\subsection{Topoi} \label{backgroundtopoi}

Topoi were defined by Grothendieck as a \textit{natural} generalization of the category of sheaves $\mathsf{Sh}(X)$ over a topological space $X$. Their geometric nature was thus the first to be explored and exploited. Yet, with time, many other properties and facets of them have emerged, making them one of the main concepts in category theory between the 80's and 90's. Johnstone, in the preface of \cite{elephant1}, gives 9 different interpretations of what a topos \textit{can be}. In fact, this multi-faced nature of the concept of topos motivates the title of his book. In this paper we will focus on two main aspects.

\begin{itemize}
  \item A topos is a (categorification of the concept of) locale;
  \item A topos is a (family of Morita-equivalent) geometric theory;
\end{itemize}

The first one draws the connection with our previous results in \cite{thgeo}, while allows us to treat a topos as placeholders for a geometric theory.

\begin{ach}
In this section by \textit{topos} we mean Grothendieck topos.
\end{ach}

\begin{defn}[Topos]
A topos $\ce$ is  lex-reflective\footnote{This means that it is a reflective subcategory and that the left adjoint preserves finite limits. Lex stands for \textit{left exact}, and was originally motivated by homological algebra.} subcategory\footnote{Up to equivalence of categories.} of a category of presheaves over a small category, $$i^*: \Set^{C^\circ} \leftrightarrows \ce : i_*.  $$
\end{defn}

\begin{defn}[Geometric morphism]
A geometric morphism of topoi $f: \ce \to \cf$ is an adjunction $f^*: \cf \leftrightarrows \ce: f_*$\footnote{Notice that $f_*$ is the right adjoint.} whose left adjoint preserves finite limits (is left exact). We will make extensive use of the following terminology:
\begin{enumerate}
  \item[$f^*$] is the inverse image functor;
  \item[$f_*$] is the direct image functor.
  \end{enumerate}
\end{defn}

\begin{defn}[$2$-category of Topoi]
The $2$-category of topoi has topoi as objects, geometric morphisms as $1$-cells and natural transformations between left adjoints as $2$-cells.
\end{defn}

\subsubsection{Literature}
There are several standard references for the theory of topoi. Most of the technical content of the paper can be understood via \cite{sheavesingeometry}, a reference that we strongly suggest to start and learn topos theory. Unfortunately, the approach of \cite{sheavesingeometry} is a bit different from ours, and even though its content is sufficient for this paper, the intuition that is provided is not $2$-categorical enough for our purposes. The reader might have to integrate with the encyclopedic \cite{elephant1,elephant2}.  A couple of constructions that are quite relevant to us are contained only in \cite{borceux_19943}, that is otherwise very much equivalent to \cite{sheavesingeometry}.

% \subsection{Site descriptions of topoi}

% The first definition of topos that has been given was quite different from the one that we have introduced. As we have mentioned, topoi were introduced as category of sheaves over a space, thus the first definition was based on a generalization of this presentation. This is the theory of sites, and the reader of \cite{sheavesingeometry} will recognize this approach in \cite{sheavesingeometry}[Chap. 3]. In a nutshell, a site $(C,J)$ is the data of a category $C$ together a notion of covering families. For example, in the case of a topological space, $C$ is the locale of open sets of $X$, and $J$ is given by the open covers. Thus, a topos can be defined to be a category of sheaves over a small site, $$ \ce \simeq \mathsf{Sh}(C,J). $$
% $\mathsf{Sh}(C,J)$ is defined as a full subcategory of $\Set^{C^{\circ}}$, which turn out to be lex-reflective. That's the technical bridge between the site-theoretic description of a topos and the one at the beginning of the section.
% Site theory is extremely useful in order to study topoi as \textit{categories}, while our approach is much more useful in order to study them as \textit{objects}. We will never use explicitly site theory in the paper, with the exception of a couple of proofs and a couple of examples.

\subsubsection{Topoi and Geometry}
It's a bit hard to convey the relationship between topos theory and geometry in a short subsection. We mainly address the reader to \cite{leinster2010informal}. Let us just mention that to every topological space $X$, one can associate its category of sheaves $\mathsf{Sh}(X)$ (and this category is a topos), moreover, this assignment is a very strong topological invariant. For this reason, the study of $\mathsf{Sh}(X)$ is equivalent to the study of $X$ from the perspective of the topologist, and is very convenient in algebraic geometry and algebraic topology. For example, the category of sets is the topos of sheaves over the one-point-space, $$ \Set \cong \mathsf{Sh}(\bullet) $$ for this reason, category-theorists sometime call $\Set$ \textit{the point}. This intuition is consistent with the fact that $\Set$ is the terminal object in the category of topoi. Moreover, as a point $p \in X$ of a topological space  $X$ is a continuous function $p: \bullet \to X$, a point of a topos $\cg$ is a geometric morphism $p: \Set \to \cg$. Parallelisms of this kind have motivated most of the definitions of topos theory and most have led to results very similar to those that were achieved in formal topology (namely the theory of locales). The class of points of a topos $\ce$ has a structure of category $\mathsf{pt}(\ce)$ in a natural way, the arrows being natural transformations between the inverse images of the geometric morphisms.

\subsubsection{Topoi and Logic}
Geometric logic and topos theory are tightly bound together. Indeed, for a geometric theory $\mathbb{T}$ it is possible to build a topos $\Set[\mathbb{T}]$ (the classifying topos of $\mathbb{T}$) whose category of points is precisely the category of models of $\mathbb{T}$,
$$ \mathsf{Mod}(\mathbb{T}) \cong \mathsf{pt}(\Set[\mathbb T]). $$
This amounts to the theory of classifying topoi \cite[Chap. X]{sheavesingeometry} and each topos classifies a geometric theory. This gives us a logical interpretation of a topos. Each topos is a \textit{geometric theory}, which in fact can be recovered from any of its sites of definition. Obviously, for each site that describes the same topos we obtain a different theory. Yet, these theories have the same category of models (in any topos). In this paper we will exploit the construction of \cite{borceux_19943} to show that to each \textit{geometric} sketch (a kind of theory), one can associate a topos whose points are precisely the models of the sketch. This is another way to say that the category of topoi can internalize a geometric logic.

\subsubsection{Special classes of topoi}
In the paper we will study some relevant classes of topoi. In this subsection we recall all of them and give a good reference to check further details. These references will be repeated in the relevant sections.

\begin{table}[!htbp]
\begin{tabular}{lllll}
\cline{2-3}
\multicolumn{1}{l|}{} & \multicolumn{1}{c|}{Topoi} & \multicolumn{1}{c|}{Reference} &  &  \\ \cline{2-3}
\multicolumn{1}{l|}{} & \multicolumn{1}{c|}{connected} & \multicolumn{1}{c|}{{\cite[C1.5.7]{elephant2}}} &  &  \\ \cline{2-3}
\multicolumn{1}{l|}{} & \multicolumn{1}{c|}{compact} & \multicolumn{1}{c|}{{\cite[C3.2]{elephant2}}} &  &  \\ \cline{2-3}
\multicolumn{1}{l|}{} & \multicolumn{1}{c|}{atomic} & \multicolumn{1}{c|}{{\cite[C3.5]{elephant2}}} &  &  \\ \cline{2-3}
\multicolumn{1}{l|}{} & \multicolumn{1}{c|}{locally decidable} & \multicolumn{1}{c|}{{\cite[C5.4]{elephant2}}} &  &  \\ \cline{2-3}
\multicolumn{1}{l|}{} & \multicolumn{1}{c|}{coherent} & \multicolumn{1}{c|}{{\cite[D3.3]{elephant2}}} &  &  \\ \cline{2-3}
\multicolumn{1}{l|}{} & \multicolumn{1}{c|}{boolean} & \multicolumn{1}{c|}{{\cite[D3.4, D4.5]{elephant2}, \cite[A4.5.22]{elephant1}}} &  &  \\ \cline{2-3}

\end{tabular}
\end{table}

\subsection{Ionads} \label{backgroundionads}

Ionads were defined by Garner in \cite{ionads}; together with our \cite{thgeo}, this is all the literature available on the topic. Garner's definition is designed to generalize the definition of topological space. Indeed a topological space $\cx$ is the data of a set (of points) and an interior operator, $$\text{Int}: 2^X \to 2^X.$$ Garner builds on the well known analogy between powerset and presheaf categories and extends the notion of interior operator to a presheaf category. The whole theory is extremely consistent with the expectations: while the poset of (co)algebras for the interior operator is the locale of open sets of a topological space, the category of coalgebras of a ionad is a topos, a natural categorification of the concept of locale. In our paper \cite{thgeo} we have provided a more refined notion of ionads which will be very important for our constructions. Garner's notion would not apply to any of our cases.

\subsection{Garner's definitions}
\begin{defn}[Ionad]
An ionad $\cx = (X, \text{Int})$ is a set $X$ together with a comonad $\text{Int}: \Set^X \to \Set^X$ preserving finite limits.
\end{defn}

\begin{defn}[Category of opens of a ionad]
The category of opens $\mathbb{O}(\cx)$ of a ionad $\cx = (X, \text{Int})$ is the category of coalgebras of $\text{Int}$. We shall denote by $U_{\cx}$ the forgetful functor $U_\cx: \mathbb{O}(\cx) \to \Set^X$.  
\end{defn}

\begin{defn}[Morphism of Ionads]
A morphism of ionads $f: \cx \to \cy$ is a couple $(f, f^\sharp)$ where $f: X \to Y$ is a set function and $f^\sharp$ is a lift of $f^*$,
\begin{center}
\begin{tikzcd}
\mathbb{O}(\cy) \arrow[rr, "f^\sharp" description] \arrow[dd, "U_\cy" description] &  & \mathbb{O}(\cx) \arrow[dd, "U_\cx" description] \\
                                                                               &  &                                                  \\
\Set^Y \arrow[rr, "f^*" description]                                           &  & \Set^X                                          
\end{tikzcd}
\end{center}
\end{defn}

\begin{defn}[Specialization of morphism of ionads]
Given two morphism of ionads $f,g: \cx \to \cy$, a specialization of morphism of ionads $\alpha: f \Rightarrow g$ is a natural transformation between $f^\sharp$ and $g^\sharp$,
\begin{center}
\begin{tikzcd}
\mathbb{O}(\cy) \arrow[r, "f^\sharp" description, bend left=35, ""{name=U, below}]
\arrow[r,"g^\sharp" description, bend right=35, ""{name=D}]
& \mathbb{O}(\cx)
\arrow[Rightarrow, "\alpha" description, from=U, to=D] \end{tikzcd}
\end{center}
\end{defn}

\begin{defn}[$2$-category of Ionads]
The $2$-category of ionads has ionads as objects, morphism of ionads as $1$-cells and specializations as $2$-cells.
\end{defn}

\begin{defn}[Bounded Ionads]
A ionad $\cx$ is bounded if $\mathbb{O}(\cx)$ is a topos.
\end{defn}

\subsection{A generalization of Garner's ionads}
In his paper, Garner mentions that in giving the definition of ionad he could have chosen a category instead of a set \cite[Rem. 2.4]{ionads}. For technical reasons we cannot be content with Garner's original definition, thus we will allow ionads over a category (as opposed to sets), even a locally small (but possibly large) one.

\begin{defn}[{\cite[Def 2.2.2]{thgeo}} Generalized Ionads]
A generalized ionad $\cx = (X, \text{Int})$ is a locally small (but possibly large) pre-finitely cocomplete category $X$ together with a lex comonad $\text{Int}: \P(X) \to \P(X)$.
\end{defn}

\begin{ach}
We will always omit the adjective \textit{generalized}.
\end{ach}

\begin{rem} The notion of generalized ionad has some delicate aspects, that we cannot discuss in this paper, we refer to our discussion in \cite[Sec. 2]{thgeo} for a complete introduction to the topic. What is important to remind is that, despite the decnical issues, the very idea of Garner remains completely available, and all the other assumptions take care of set-theoretic subtleties generated by the fact that we study possibly large categories insted of crude sets.
\end{rem}

\begin{rem} The main results of \cite[Thm. 3.2.6 and 4.0.3]{thgeo} provide an adjunction between the $2$-category of ionads and the $2$-category of topoi which categories the adjunction between topological spaces and locales. Since the latter can be seen as a completeness theorem for propositional logic, we will use \cite[Thm. 3.2.6 and 4.0.3]{thgeo} in this paper to deduce completeness-like theorems for geometric logic.
\end{rem}

\section{Generalized axiomatizations and the Scott construction}  \label{logicgeneralizedaxiom}

\begin{rem}\label{groups}
Let $\mathsf{Grp}$ be the category of groups and  $\mathsf{U}: \mathsf{Grp} \to \Set$ be the forgetful functor. The historical starting point of a categorical understanding of universal algebra was precisely that one can recover the (a maximal presentation of) the algebraic theory of groups from $\mathsf{U}$. Consider all the natural transformations of the form \[\mu: \mathsf{U}^n \Rightarrow \mathsf{U}^m, \]
these can be seen as implicitly defined operations of groups. If we gather these operations in an equational theory $\mathbb{T}_\mathsf{U}$, we see that the functor $\mathsf{U}$ lifts to the category of models $\mathsf{Mod}(\mathbb{T}_\mathsf{U})$ as indicated by the diagram below.

\begin{center}
\begin{tikzcd}
\mathsf{Grp} \arrow[rdd, "\mathsf{U}" description] \arrow[r, dotted] & \mathsf{Mod}(\mathbb{T}_\mathsf{U}) \arrow[dd, "|-|" description] \\
                                                                     &                                                          \\
                                                                     & \Set                                                    
\end{tikzcd}
\end{center}
It is a quite classical result that the comparison functor above is fully faithful and essentially surjective, thus we have axiomatized the category of groups (probably with a non minimal family of operations).
\end{rem}

\begin{rem}\label{lawvereliterature}
The idea above was introduced in Lawvere's PhD thesis \cite{lawvere1963functorial} and later developed in great generality by Linton \cite{10.1007/978-3-642-99902-4_3,10.1007/BFb0083080}. The interested reader might find interesting \cite[Chap. 3]{adamekrosicky94} and the expository paper \cite{HYLAND2007437}.   Nowadays this is a standard technique in categorical logic and some generalizations of it were presented in \citep{infinitarylang} by Rosický and later again in \citep[Rem. 3.5]{LB2014}.
\end{rem}

\begin{rem}[Lieberman-Rosický construction]
In  \citep[Rem. 3.5]{LB2014} given a couple $(\ca, \mathsf{U})$ where $\ca$ is an a accessible category with directed colimits together with a faithful functor $\mathsf{U}: \ca \to \Set$ preserving directed colimits, the authors form a category $\mathbb{U}$ whose objects are finitely accessible sub-functors of $\mathsf{U}^n$ and whose arrows are natural transformations between them. Of course there is a naturally attached signature $\Sigma_U$ and a naturally attached first order theory $\mathbb{T}_\mathsf{U}$.  In the same fashion as the previous remarks one finds a comparison functor $\ca \to \Sigma_\mathsf{U}\text{-Str}$. In \citep[Rem. 3.5]{LB2014} the authors stress that is the most natural candidate to axiomatize $\ca$. A model of $\mathbb{T}_\mathsf{U}$ is the same as a functor $\mathbb{U} \to \Set$ preserving products and subobjects. Of course the functor $\ca \to \Sigma_\mathsf{U}\text{-Str}$ factors through $\text{Mod}(\mathbb{U})$ (seen as a sketch) \[l: \ca \to \text{Mod}(\mathbb{U}),\] but in \citep[Rem. 3.5]{LB2014} this was not the main concern of the authors.
\end{rem}

\begin{rem}[Rosický's remark] 
Rem. \ref{groups} ascertains that the collection of functor $\{ \mathsf{U}^n\}_{n \in \mathbb{N}}$, together with all the natural transformations between them, retains all the informations about the category of groups. Observe that in this specific case, the functors $ \mathsf{U}^n$ all preserve directed colimits, because finite limits commute with directed colimits. More generally, when $\ca$ does not come equipped with a special forgetful functor, or simply we don't want to choose a specific one, we could follow the general strategy of the remarks above and collect \text{all} the functors preserving directed colimits into $\Set$ in a category. This is the Scott construcion.
\end{rem}

 \begin{con}[The Scott construction]\label{defnS}
We recall the construction of $\mathsf{S}$ from \cite{simon} and \cite{thcat}. Let $\ca$ be an accessible category with directed colimits.  $\mathsf{S}(\ca)$ is defined as the category the category of functors preserving directed colimits into sets. \[\mathsf{S}(\ca) = \text{Acc}_{\omega}(\ca, \Set).\]
The category $\mathsf{S}(\ca)$ is a Grothendieck topos and thus can be seen as a geometric theory. Following the discussion above, this is a candidate geometric axiomatization of $\ca$. In \cite{thcat} we study the Scott construction and show that it is functorial, providing a left adjoint for the functor of points.
\end{con}

\begin{rem}[The functor $\mathsf{pt}$]\label{pt}
The functor of points $\mathsf{pt}: \text{Topoi} \to \text{Acc}_\omega$ belongs to the literature since quite some time, $\mathsf{pt}$ is the covariant hom functor $\text{Topoi}(\Set, - )$. It maps a Grothendieck topos $\cg$ to its category of points, \[\cg \mapsto \text{Cocontlex}(\cg, \Set).\]
Of course given a geometric morphism $f: \cg \to \ce$, we get an induced morphism $\mathsf{pt}(f): \mathsf{pt}(\cg) \to \mathsf{pt}{(\ce)}$ mapping $p^* \mapsto p^* \circ f^*$. The fact that $\text{Topoi}(\Set, \cg)$ is an accessible category with directed colimits appears in the classical reference by Borceux as \cite[Cor. 4.3.2]{borceux_19943}, while the fact that $\mathsf{pt}(f)$ preserves directed colimits follows trivially from the definition. 
\end{rem}

\begin{rem}
When we idenitify the category of topoi with a localization of the category of geometric theories, the functor of points is computing the (set theoretic) models of the theory classified by the topos. Being a right adjoint, it is coherent with the intuition that its left adjoint computes the \textit{free theory} over an accessible category with directed colimits.
\end{rem}

\begin{thm}[{\cite[Prop. 2.3]{simon},\cite[Thm. 2.1]{thcat}} The Scott adjunction]\label{scottadj}
The $2$-functor of points $\mathsf{pt} :\text{Topoi} \to \text{Acc}_{\omega} $ has a left biadjoint $\mathsf{S}$, yielding the Scott biadjunction, $$\mathsf{S} : \text{Acc}_{\omega} \leftrightarrows \text{Topoi}: \mathsf{pt}. $$
\end{thm}

\begin{rem}[Rosický's remark] 
 Going back to Rosický-Lieberman construction, the previous discussion implies that the small category $\{ \mathsf{U}^n\}_{n \in \mathbb{N}}$ is a full subcategory of the Scott topos of the category of groups. In fact the vocabulary of the theory that we used to axiomatize the category of groups is made up of symbols coming from a full subcategory of the Scott topos.
\end{rem}

\begin{rem}[Generalized axiomatizations]
The generalized axiomatization of Lieberman and Rosický amounts to a sketch $\mathbb{U}$. As we mentioned, there exists an obvious inclusion of $\mathbb{U}$ in the Scott topos of $\ca$, $$i: \mathbb{U} \to \mathsf{S}(\ca)$$ which is a flat functor because finite limits in $\mathsf{S}(\ca)$ are computed pointwise in $\Set^\ca$. Thus, every point $p: \Set \to \mathsf{S}(\ca)$ induces a model of the sketch $\mathbb{U}$ by  composition,
$$i^*: \mathsf{pt}(\mathsf{S}\ca) \to \text{Mod}(\mathbb{U})$$
$$p \mapsto  p^* \circ i.$$
In particular this shows that the unit of the Scott adjunction lifts the comparison functor between $\ca$ and $\text{Mod}(\mathbb{U})$ along $i^*$ and thus the Scott topos provides a \textit{sharper} axiomatization of $\mathbb{T}_\mathsf{U}$.

\begin{center}
\begin{tikzcd}
                                                         & \ca \arrow[ldd, "\eta_\ca" description, bend right] \arrow[rdd, "l" description, bend left] &                        \\
                                                         &                                                                                             &                        \\
\mathsf{pt}\mathsf{S}(\ca) \arrow[rr, "i^*" description] &                                                                                             & \text{Mod}(\mathbb{U})
\end{tikzcd}
\end{center}
\end{rem}

\begin{rem}[Faithful functors are likely to generate the Scott topos] Yet, it should be noticed that when $\mathbb{U}$ is a generator in $\mathsf{S}(\ca)$, the functor $i^*$ is an equivalence of categories. As unlikely as it may sound, in all the examples that we can think of, a generator of the Scott topos is always given by a faithful forgetful functor $\mathsf{U}: \ca \to \Set$. This phenomenon is so pervasive that the author has believed for quite some time that an object in the Scott topos $\mathsf{S}(\ca)$ is a generator if and only if it is faithful and conservative. We still lack a counterexample, or a theorem proving such a statement.
\end{rem}
 
\section{Classifying topoi} \label{logicclassifyingtopoi}

This section is devoted to specifying the connection between Scott topoi, Isbell topoi and classifying topoi. Recall that for a geometric theory $\mathbb{T}$, a classifying topos $\Set[\mathbb{T}]$ is a topos representing the functor of models in topoi, \[ \mathsf{Mod}_{(-)}(\mathbb{T}) \cong \text{Topoi}(-, \Set[\mathbb{T}]).\]  The theory of classifying topoi allows us to internalize geometric logic in the internal logic of the $2$-category of topoi. The reader that is not familiar with the theory of classifying topoi is encouraged to check the Appendix.

\subsection{Categories of models, Scott topoi and classifying topoi} 

The Scott topos $\mathsf{S}(\mathsf{Grp})$ of the category of groups is $\Set^{\mathsf{Grp}_{\omega}}$, this follow from \cite[Rem 2.13]{thcat} and applies to $\mathsf{Mod}(\mathbb{T})$ for every Lawvere theory $\mathbb{T}$. It is well known that $\Set^{\mathsf{Grp}_{\omega}}$ is also the classifying topos of the theory of groups. This section is devoted to understating if this is just a coincidence, or if the Scott topos is actually related to the classifying topos.

\begin{rem}
Let $\ca$ be an accessible category with directed colimits. In order to properly ask the question \textit{is $\mathsf{S}(\ca)$ the classifying topos?}, we should answer the question \textit{the classifying topos of what?} Indeed $\ca$ is just a category, while one can compute classifying topoi of theories. Our strategy is to introduce a quite general notion of theory that fits in the following diagram,

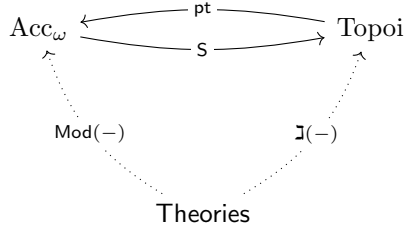
\begin{center}
\begin{tikzcd}
\text{Acc}_\omega \arrow[rr, "\mathsf{S}" description, bend right=10] &                                                                                      & \text{Topoi} \arrow[ll, "\mathsf{pt}" description, bend right=10] \\
                                                   &                                                                                      &                                                \\
                                                   &                                                                                      &                                                \\
                                                   & \mathsf{Theories} \arrow[dotted, luuu, "\mathsf{Mod}(-)" description, bend left=20] \arrow[dotted, ruuu, "\gimel(-)" description, bend right=20] &                                               
\end{tikzcd}
\end{center}

in such a way that:

\begin{enumerate}
  \item $\gimel(\mathbb{T})$ gives the classifying topos of $\mathbb{T}$;
  \item $\mathsf{Mod}(-) \cong \mathsf{pt} \gimel (-)$.
\end{enumerate} 

In this new setting we can reformulate our previous discussion in the following mathematical question: \[\gimel(-) \stackrel{?}{\cong} \mathsf{S} \mathsf{Mod}(-).\]

\end{rem}

\begin{rem}[Geometric Sketches]
The notion of theory that we plan to use is that of geometric sketch. The category of (small) sketches was described in \cite[3.1]{Makkaipare}, while a detailed study of geometric sketches was conducted in \cite{adamek_johnstone_makowsky_rosicky_1997,Admek1996OnGA}.

\begin{center}
\begin{tikzcd}
\text{Acc}_\omega \arrow[rr, "\mathsf{S}" description, bend right=10] &                                                                                      & \text{Topoi} \arrow[ll, "\mathsf{pt}" description, bend right=10] \\
                                                   &                                                                                      &                                                \\
                                                   &                                                                                      &                                                \\
                                                   & \mathsf{GSketches} \arrow[luuu, "\mathsf{Mod}(-)" description, bend left=20] \arrow[ruuu, "\gimel(-)" description, bend right=20] &                                               
\end{tikzcd}
\end{center}
\end{rem}

\begin{rem}
Following \cite{Makkaipare}, there exists a natural way to generate a sketch from any accessible category. This construction, in principle, gives even  a left adjoint for the functor $\mathsf{Mod}(-)$, but does land in large sketches. Thus it is indeed true that for each accessible category there exist a sketch (a theory) canonically associated to it. We do not follow this line because the notion of large sketch, from a philosophical perspective, is a bit unnatural. Syntax should always be very frugal. From an operational perspective, presentations should always be as small as possible. It is possible to cut down the size of the sketch, but this construction cannot be defined functorially on the whole category of accessible categories with directed colimits. Since elegance and naturality is one of the main motivations for this treatment of syntax-semantics dualities, we decided to avoid any kind of non-natural construction.
\end{rem}

\begin{rem}
Geometric sketches contain  coherent  sketches. In the dictionary between logic and geometry that is well motivated in the indicated papers (\cite{adamek_johnstone_makowsky_rosicky_1997,Admek1996OnGA}) these two classes correspond respectively to geometric and coherent theories. The latter essentially contain all first order theories via the process of Morleyzation. These observations make our choice of geometric sketches a very general notion of theory and makes us confident that it's a good notion to look at.
\end{rem}

 We now proceed to describe the two functors labeled with the name of $\mathsf{Mod}$ and $\gimel$.

\begin{rem}[Mod]
This 2-functor is very easy to describe. To each sketch $\cs$ we associate its category of Set-models, while it is quite evident that a morphism of sketches induces by composition a functor preserving directed colimits (see Sec. \ref{backgroundsketches} in the Background section).
\end{rem}

\begin{con}[$\gimel$]
The topos completion of a geometric sketch is a highly nontrivial object to describe. Among the possible constructions that appear in the literature, we refer to \citep[4.3]{borceux_19943}. Briefly, the idea behind this construction is the following. 

\begin{enumerate} 
  \item By {\citep[4.3.3]{borceux_19943}}, every sketch $\cs$ can be completed to a sketch $\bar{\cs}$ whose underlying category is cartesian.
  \item By {\citep[4.3.6]{borceux_19943}}, this construction is functorial and does not change the model of the sketch in any Grothendieck topos.
  \item By {\citep[4.3.8]{borceux_19943}}, the completion of the sketch has a natural topology $\bar{J}$.
  \item The correspondence $\cs \mapsto \bar{\cs} \mapsto (\bar{S}, \bar{J})$ transforms geometric sketches into sites and morphism of sketches into morphism of sites.
  \item We compute sheaves over the site $(\bar{S}, \bar{J})$. 
  \item Define $\gimel$ to be $\cs \mapsto \bar{\cs} \mapsto (\bar{S}, \bar{J}) \mapsto \mathsf{Sh}(\bar{S}, \bar{J})$.
\end{enumerate}
\end{con}

\begin{rem}
While \citep[4.3.6]{borceux_19943} proves that  $\mathsf{Mod}(-) \simeq \mathsf{pt} \gimel (-)$, and \citep[4.3.8]{borceux_19943} prove that $\gimel(\cs)$ is the classifying topos of $\cs$ among Grothendieck topoi, the main question of this section remains completely open, is $\gimel(\cs)$ isomorphic to the Scott topos $\mathsf{S} \mathsf{Mod}(-)$ of the category of Set models of $\cs$? We answer this question with the following theorem.
\end{rem}

\begin{thm}\label{classificatore}
If the counit $\epsilon_{\gimel(\cs)}$ of the Scott adjunction is an equivalence of categories on $\gimel(\cs)$, then $\gimel(\cs)$ coincides with $\mathsf{S} \mathsf{Mod}(\cs)$.
\end{thm}
\begin{proof}
We introduced enough technology to make this proof incredibly slick. Recall the counit \[\mathsf{S} \mathsf{pt} (\gimel(\cs)) \to \gimel(\cs) \] and assume that it is an equivalence of categories. Now, since $\mathsf{Mod}(-) \simeq \mathsf{pt} \gimel (-)$, we obtain that \[\gimel(\cs) \simeq \mathsf{S}\mathsf{Mod}(\cs),\] which indeed it our thesis.
\end{proof}

\begin{rem}
\cite[Thm 4.1.3 and 4.3.3]{thgeo}  characterize those topoi for which the counit is an equivalence of categories, providing a full description of those geometric sketches for which $\gimel(\cs)$ coincides with $\mathsf{S} \mathsf{Mod}(\cs)$. 
\end{rem}

\subsection{Ionads of models, Isbell topoi and classifying topoi}

Indeed the main result of this section up to this point has been partially unsatisfactory. As happens sometimes, the answer is not as nice as expected because the question in the first place did not take in consideration some relevant factors. The category of models of a sketch does not retain enough information on the sketch. Fortunately, we will show that every sketch has a ionad of models (not just a category) and the category of opens of this ionad is a much better approximation of the classifying topos.
In this subsection, we switch diagram of study to the one below.

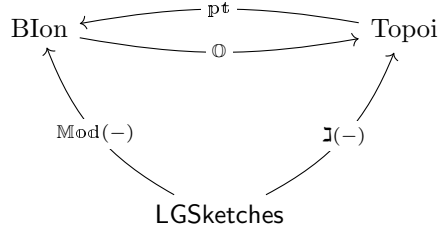
\begin{center}
\begin{tikzcd}
\text{BIon} \arrow[rr, "\mathbb{O}" description, bend right=10] &                                                                                      & \text{Topoi} \arrow[ll, "\mathbbm{pt}" description, bend right=10] \\
                                                   &                                                                                      &                                                \\
                                                   &                                                                                      &                                                \\
                                                   & \mathsf{LGSketches} \arrow[luuu, "\mathbb{M}\mathbbm{od}(-)" description, bend left=20] \arrow[ruuu, "\gimel(-)" description, bend right=20] &                                               
\end{tikzcd}
\end{center}

Of course, in order to study it, we need to introduce all its nodes and legs. We should say what we mean by $\mathsf{LGSketches}$ and $\mathbb{M}\mathbbm{od}(-)$. The adjunction $\mathbb{O} \dashv \mathbbm{pt}$ was introduced and studied in \cite{thgeo} and it relates topoi to bounded ionads, we refer to \cite[Sec. 3]{thgeo} for the construction, while an introduction to ionads can be found in the Appendix.

Whatever $\mathsf{LGSketches}$ and $\mathbb{M}\mathbbm{od}(-)$ will be, the main point of the section is to show that this diagram fixes the one of the previous section, in the sense that we will obtain the following result.

 \begin{thm*} The following are equivalent:
 \begin{itemize} 
  \item $\gimel(\cs)$ has enough points;
  \item $\gimel(\cs)$ coincides with $\mathbb{O}\mathbb{M}\mathbbm{od}(\cs)$.
 \end{itemize}
 \end{thm*}

 We decided to present this theorem separately from the previous one because indeed a ionad of models is a much more complex object to study than a category of models, thus the results of the previous section are indeed very interesting, because easier to handle.

\begin{exa}[Motivating ionads of models: Ultracategories]
We are not completely used to thinking about ionads of models. Indeed a (bounded) ionad is quite complex data, and we do not completely have a logical intuition on its interior operator. \textit{In which sense does the interior operator equip a category of models with a topology?} One very interesting example, that hasn't appeared in the literature to our knowledge is the case of ultracategories. Ultracategories where introduced by Makkai in \cite{AWODEY2013319} and later simplified by Lurie in \cite{lurieultracategories}. These objects are the data of a category $\ca$ together with an ultrastructure, that is a family of functors \[\int_X:\beta(X) \times \ca^X \to \ca. \] We redirect to \cite{lurieultracategories} for the precise definition. In a nutshell, each of these functors $\int_X$ defines a way to compute the ultraproduct of an $X$-indexed family of objects along some ultrafilter. Of course there is a notion of morphism of ultracategories, namely a functor $\ca \to \cb$ which is compatible with the ultrastructure  \cite[Def. 1.41]{lurieultracategories}. Since the category of sets has a natural ultrastructure, for every ultracategory $\ca$ one can define $\text{Ult}(\ca, \Set)$ which obviously sits inside $\Set^\ca$. Lurie observes that the inclusion \[\iota: \text{Ult}(\ca, \Set) \to \Set^\ca\] preserves all colimits \cite[War. 1.4.4]{lurieultracategories}, and in fact also finite limits (the proof is the same). In particular, when $\ca$ is accessible and every ultrafunctor is accessible, the inclusion $\iota: \text{Ult}(\ca, \Set) \to \Set^\ca$ factors through $\P(\ca)$ and thus the ultrastructure over $\ca$ defines a idempotent lex comonad over $\P(\ca)$ by the adjoint functor theorem. This shows that every (good enough) accessible ultracategory yields a ionad, which is also \text{compact} in the sense that its category of opens is a compact (coherent) topos. This example is really a step towards a categorified Stone duality involving compact ionads and boolean topoi.
\end{exa} 

\subsubsection{$\mathsf{LGSketches}$ and $\mathbb{M}\mathbbm{od}(-)$}

\begin{defn}
A geometric sketch $\mathcal{S}$ is lex if its underlying category has finite limits and every limiting cone is in the limit class.
\end{defn}

\begin{rem}[Lex sketches are \textit{enough}]
\cite[4.3.3]{borceux_19943} shows that every geometric sketch can be replaced with a lex geometric sketch in such a way that the underlying category of models, and even the classifying topos, does not change. In this sense this full subcategory of geometric sketches is as expressive as the whole category of geometric sketches.
\end{rem}

\begin{prop}[$\mathbb{M}\mathbbm{od}(-)$ on objects]

Every lex geometric sketch $\mathcal{S}$ induces a ionad $\mathbb{M}\mathbbm{od}(\mathcal{S})$ over its category of models $\mathsf{Mod}(\mathcal{S})$.
\end{prop}
\begin{proof}
The underlying category of the ionad $\mathbb{M}\mathbbm{od}(\mathcal{S})$ is $\mathsf{Mod}(\mathcal{S})$. We must provide an interior operator (a lex comonad), $$\text{Int}_{\cs}: \P({\mathsf{Mod}(\mathcal{S})}) \to \P({\mathsf{Mod}(\mathcal{S})}). $$ In order to do so, we consider  the evaluation pairing $\mathsf{eval}: S \times \mathsf{Mod}(\cs) \to \Set$ mapping $(s,p) \mapsto p(s)$. Let $\mathsf{ev}: \cs \to \Set^{\mathsf{Mod}(S)}$ be its mate. Similarly to \cite[Con. 3.2.3]{thgeo}, such functor takes values in $\P(\mathsf{Mod}(S))$. Because $\mathcal{S}$ is a lex sketch, this functor must preserve finite limits. Indeed, \[\mathsf{ev}(\lim s_i)(-) \cong (-)(\lim s_i) \cong \lim ((-)(s_i)) \cong  \lim \mathsf{ev}(s_i)(-).\]
Now, the left Kan extension $\lan_y \mathsf{ev}$ (see diagram below) is left exact because $\P(\mathsf{Mod}(S))$ is an infinitary pretopos and $\mathsf{ev}$ preserves finite limits. 
\begin{center}
\begin{tikzcd}
S \arrow[rr, "\mathsf{ev}" description] \arrow[ddd, "y" description]               &  & \P(\mathsf{Mod}(S)) \\
                                                                                     &  &                          \\
                                                                                     &  &                          \\
\Set^{S^\circ} \arrow[rruuu, "\lan_y \mathsf{ev}" description, dashed, bend right] &  &                         
\end{tikzcd}
\end{center}
Moreover it is cocontinuous because of the universal property of the presheaf construction. Because $\Set^{S^\circ}$ is a total category,  $\lan_y \mathsf{ev}$ must have a right adjoint (and it must coincide with $\lan_{\mathsf{ev}} y$). The induced comonad must be left exact, because the left adjoint is left exact. Define \[\text{Int}_{\cs}:=\lan_y \mathsf{ev} \circ \lan_{\mathsf{ev}} y. \]
Observe that $\text{Int}_{\cs}$ coincides with the density comonad of $\mathsf{ev}$ by \cite[A.7]{liberti2019codensity}. Such result dates back to \cite{appelgate1969categories}.
\end{proof} 

\begin{rem}[$\mathbb{M}\mathbbm{od}(-)$ on morphism of sketches]
This definition will not be given explicitly: in fact we will use the following remark to show that the ionad above is isomorphic to the one induced by $\gimel(\cs)$, and thus there exists a natural way to define $\mathbb{M}\mathbbm{od}(-)$ on morphisms. 
\end{rem}
 
\subsubsection{Ionads of models and theories with enough points}

\begin{rem}
In the main result of the previous section, a relevant rôle was played by the fact that $\mathsf{pt}\gimel \simeq \mathsf{Mod}.$ The same must be true in this one. Thus we should show that $\mathbbm{pt}\gimel \simeq \mathbb{M}\mathbbm{od}.$ Indeed we only need to show that the interior operator is the same, because the underlying category is the same by the discussion in the previous section. 
\end{rem}

\begin{prop}
\[\mathbbm{pt} \circ \gimel \simeq \mathbb{M}\mathbbm{od}.\]
\end{prop}
\begin{proof}

Let $\cs$ be a lex geometric sketch. Of course there is a map $j: S \to \gimel{\cs}$, because $S$ is a site of definition of $\gimel{\cs}$. Moreover, $j$ is obviously dense. In particular the evaluation functor that defines the ionad $\mathbbm{pt} \circ \gimel$ given by $ev^*: \gimel(\cs) \to \P({\mathsf{pt} \circ \gimel(\cs)})$ is uniquely determined by its composition with $j$. This means that the comonad $ev^*ev_*$ is isomorphic to the density comonad of the composition $ev^* \circ j$. Indeed, \[ev^*ev_* \cong \lan_{ev^*} ev^* \cong \lan_{ev^*} (\lan_j(ev^*j)) \cong \lan_{ev^*j}(ev^*j).\] Yet, $ev^*j$ is evidently $\mathsf{ev}$, and thus $ev^* ev_* \cong \text{Int}_\cs$ as desired.

\end{proof}

 \begin{thm} \label{isbellclassificatore} The following are equivalent:
 \begin{itemize} 
  \item $\gimel(\cs)$ has enough points;
  \item $\gimel(\cs)$ coincides with $\mathbb{O}\mathbb{M}\mathbbm{od}(\cs)$.
 \end{itemize}
 \end{thm}
 \begin{proof}
 By \cite[Thm. 4.0.3]{thgeo}, $\gimel(S)$ has enough points if and only if the counit of the categorified Isbell duality   $\rho: \mathbb{O}\mathbbm{pt}(\gimel)(\cs) \to \cs$ is an equivalence of topoi. Now, since $\mathbbm{pt} \circ \gimel \cong \mathbb{M}\mathbbm{od}$, we obtain the thesis.
 \end{proof}

 \section{Abstract elementary classes and locally decidable topoi} \label{logicaec}

\subsection{A general discussion}
This section is dedicated to the interaction between Abstract elementary classes and the Scott adjunction. Abstract elementary classes were introduced in the 70's by Shelah as a framework to encompass infinitary logics within the language of model theorist. In principle, an abstract elementary class $\ca$ should look like the category of models of a first order infinitary theory whose morphisms are elementary embeddings. The problem of relating abstract elementary classes and accessible categories has been tackled by Lieberman \citep{L2011}, and Beke and Rosický \cite{aec}, and lately has attracted the interest of model theorists such as Vasey, Boney and Grossberg \citep{everybody}.  There are many partial, even very convincing results, in this characterization. Let us recall at least one of them. For us, this characterization will be the definition of abstract elementary class.

 \begin{thm}[{\citep[5.7]{aec}}]\label{AEC}  A category $\ca$ is equivalent to an abstract elementary class if and only if it is an accessible category with directed colimits, whose morphisms are monomorphisms and which admits a full with respect to isomorphisms and nearly full embedding $U$ into a finitely accessible category preserving directed colimits and monomorphisms.
 \end{thm}

 \begin{defn} 
 A functor $\mathsf{U} : \ca \to \cb$ is nearly full if, given a commutative diagram,
 \begin{center}
 \begin{tikzcd}
\mathsf{U}(a) \arrow[rrd, "\mathsf{U}(f)" description] \arrow[dd, "h" description] &  &  \\
 &  & \mathsf{U}(c) \\
\mathsf{U}(b) \arrow[rru, "\mathsf{U}(g)" description] &  & 
\end{tikzcd}
 \end{center}
 in $\cb$, there is a map $\bar{h}$ in $\ca$ such that $h = \mathsf{U}(\bar{h})$ and $g\bar{h}=f$. Observe that when $\mathsf{U}$ is faithful such a filling has to be unique.
 \end{defn}
 
 \begin{rem}
In some reference the notion of nearly-full functor was called coherent, referring directly to the \textit{coherence axiom} of AECs that it incarnates. The word coherent is overloaded in category theory, and thus we do not adopt this terminology, but nowadays it is getting more and more common. 
\end{rem}

 \begin{exa}[$\mathsf{pt}(\ce)$ is likely to be an AEC]\label{esempiobase}
 Let $\ce$ be a Grothendieck topos and $f^*: \Set^C \leftrightarrows \ce :f_*$ a presentation of $\ce$. By a combination of \cite[\text{Prop.} 4.2]{thcat} and \cite[Rem. 2.12]{thcat}, applying the functor $\mathsf{pt}$ we get a fully faithful functor  \[\mathsf{pt}(\ce) \stackrel{}{\to} \mathsf{pt}(\Set^C) \stackrel{}{\cong} \mathsf{Ind}(C) \]
 into a finitely accessibly category. Thus when every map in $\mathsf{pt}(\ce)$ is a monomorphism we obtain that $\mathsf{pt}(\ce)$ is an AEC via Thm. \ref{AEC}. We will see in the next section (Thm. \ref{LDCAECs}) that this happens when $\ce$ is locally decidable; thus the category of points of a locally decidable topos is always an AEC.
 \end{exa}

 \begin{exa}[$\eta_\ca$ behaves nicely on AECs]
When $\ca$ is an abstract elementary class, the unit of the Scott adjunction $\eta_\ca: \ca \to \mathsf{pt}\mathsf{S}(\ca)$ is faithful and iso-full. This follows directly from \cite[Prop 4.13]{thcat}.
 \end{exa}

 \begin{rem}
 Even if this is the sharpest (available)  categorical characterization of AECs it is not hard to see how unsatisfactory it is. Among the most evident problems, one can see that it is hard to provide a categorical \textit{understanding} of nearly full and full with respect to isomorphisms. Of course, an other problem is that the list of requirements is pretty long and very hard to check: \textit{when does such a $U$ exist?}
 \end{rem}

It is very hard to understand when such a pseudo monomorphism exists. That is why it is very useful to have a testing lemma for its existence.

\begin{thm}[Testing lemma]
Let $\ca$ be an object in $\text{Acc}_{\omega}$ where every morphism is a monomorphism. If $\eta_\ca$ is a nearly-full pseudo monomorphism, then $\ca$ is an AEC.
\end{thm}
\begin{proof}
The proof is relatively easy, choose a presentation $f^*: \Set^C \leftrightarrows \mathsf{S}(\ca): f_*$ of $\mathsf{S}(\ca)$. Now in

\[\ca \stackrel{\eta_\ca}{\to} \textsf{pt}\textsf{S}(\ca) \to \textsf{pt}(\Set^{C}) {\cong} \mathsf{Ind}(C),\]

by a combination of \cite[\text{Prop.} 4.2]{thcat} and \cite[Rem. 2.12]{thcat}, the composition is a faithful and nearly full functor preserving directed colimits from an accessible category to a finitely accessible category, and thus $\ca$ is an AEC because of Thm. \ref{AEC}.
\end{proof}

\subsection{Locally decidable topoi and AECs}

The main result of this subsection relates locally decidable topoi to AECs. The full subcategory of $\text{Acc}_\omega$ whose objects are AECs will be indicated by $\text{AECs}$. As in the previous sections, let us give the precise statement and then discuss it in better detail.

\begin{thm}\label{LDCAECs} The Scott adjunction restricts to locally decidable topoi and AECs.

\[\mathsf{S}: \text{AECs} \leftrightarrows \text{LDTopoi}: \mathsf{pt}\]
\end{thm}

\subsubsection{Locally decidable topoi}

The definition of locally decidable topos will appear obscure at first sight.

\begin{defn}[Decidable object]
An object $e$ in a topos $\ce$ is decidable if the diagonal map $e \to e \times e$ is a complemented subobject.
\end{defn}

\begin{defn}[Locally decidable topos]
An object $e$ in a topos $\ce$ is called locally decidable iff there is an epimorphism $e' \twoheadrightarrow e$ such that $e'$ is a decidable object.  $\ce$ is locally decidable if every object is locally decidable.
\end{defn}

In order to make the definition above clear we should really define decidable objects and discuss their meaning. This is carried out in the literature and it is not our intention to recall the whole theory of locally decidable topoi. Let us instead give the following characterization, that we may take as a definition.

\begin{thm}[{\citep[C5.4.4]{elephant2}}, Characterization of loc. dec. topoi] The following are equivalent:
\begin{enumerate}
  \item $\ce$ is locally decidable;
  \item there exists a site $(C,J)$ of presentation where every map is epic;
  \item there exists a localic geometric morphism into a Boolean topos.
\end{enumerate}
\end{thm}

\begin{rem}
Recall that a localic topos $\ce$ is a topos of sheaves over a locale. The theorem above (which is due to Freyd \cite{Aspects}) shows that a locally decidable topos is still a topos of sheaves over a locale, but the locale is not in $\Set$. It is instead in some boolean topos. A boolean topos is the closest kind of topos we can think of to the category of sets itself. For more details, we redirect the reader to the Background section, where we give references to the literature.
\end{rem}

%\subsubsection{Proof of Thm. \ref{LDCAECs}}

\begin{proof}[Proof of Thm. \ref{LDCAECs}]
\begin{itemize}
  \item[]
  \item Let $\ce$ be a locally decidable topos. By Exa. \ref{esempiobase}, it is enough to show that every map in $\mathsf{pt}(\ce)$ is a monomorphism. This is more or less a folklore result, let us give the shortest path to it given our technology. Recall that one of the possible characterization of a locally decidable topos is that it has a localic geometric morphism into a boolean topos $\ce \to \cb$. If $\cb$ is a boolean topos, then every map in $\mathsf{pt}(\cg)$ is a monomorphism \cite[D1.2.10, last paragraph]{elephant2}. Now, the induce morphism below,

\[ \textsf{pt}(\ce) \to \textsf{pt}(\cb),\]
is faithful by Prop. \cite[Prop. 4.4]{thcat}. Thus every map in $\textsf{pt}\ce$ must be a monomorphism.
  \item Let's show that for an accessible category with directed colimits $\ca$, its Scott topos is locally decidable. By \citep[C5.4.4]{elephant2}, it's enough to prove that $\mathsf{S}\ca$ has a site where every map is an epimorphism. Using \cite[Rem. 2.9]{thcat}, $\ca_\kappa^{\circ}$ is a site of definition of $\mathsf{S}\ca$, and since every map in $\ca$ is a monomorphism, every map in $\ca_\kappa^{\circ}$ is epic.
  \end{itemize}
\end{proof}

The previous theorem admits an even sharper version.

\begin{thm} Let $\ca$ be an accessible category with directed colimits and a faithful functor $\mathsf{U}: \ca \to \Set$ preserving directed colimits. If $\mathsf{S}\ca$ is locally decidable, then  every map in $\ca$ is a monomorphism.
\end{thm}
\begin{proof}

\begin{enumerate}
  \item[]
\item[Step 1]  If $\cg$ is a boolean topos, then every map in $\mathsf{pt}(\cg)$ is a monomorphism \cite[D1.2.10, last paragraph]{elephant2}.
\item[Step 2] Recall that one of the possible characterization of a locally decidable topos is that it has a localic geometric morphism into a boolean topos $\mathsf{S}(\ca) \to \cg$.
\item[Step 3]  In the following diagram
\[\ca \stackrel{\eta_\ca}{\to} \textsf{pt}\textsf{S}(\ca) \stackrel{}{\to} \textsf{pt}(\cg),\]
the composition is a faithful functor by \cite[Prop. 4.4 and 4.13]{thcat}. Thus $\ca$ has a faithful functor into a category where every map is a monomorphism. As a result every map in $\ca$ is a monomorphism.
\end{enumerate}
\end{proof}

\begin{rem}The following corollary gives a complete characterization of those continuous categories that are abstract elementary classes. Recall that continuous categories were defined in \cite{cont} in analogy with continuous posets in order to study exponentiable topoi. Among the possible characterizations, a category is continuous if and only if it is a reflective subcategory of a finitely accessible category whose right adjoint preserve directed colimits. We discussed continuous categories in the first section of \cite{thcat}.
\end{rem}
 
\begin{cor}[Continuous categories and AECs]
Let $\ca$ be a continuous category. The following are equivalent:
\begin{enumerate}
  \item $\ca$ is an AEC.
  \item Every map in $\ca$ is a monomorphism.
  \item $\mathsf{S}(\ca)$ is locally decidable.
  \end{enumerate}
\end{cor}
\begin{proof}
Since it's a split subobject in $\text{Acc}_{\omega}$ of a finitely accessible category, the hypotheses of \citep[5.7]{aec} are met.
\end{proof}

\section{Categories of saturated objects, atomicity and categoricity} \label{logicsaturatedobjects}

\begin{rem}
In this section we define categories of saturated objects and study their connection with atomic topoi and categoricity. The connection between atomic topoi and categoricity was pointed out in \citep{Caramelloatomic}. This section corresponds to a kind of syntax-free counterpart of \citep{Caramelloatomic}. In the definition of \textit{category of saturated objects} we axiomatize the relevant properties of the inclusion $\iota: \Set_{\kappa} \to \Set$ and we prove the following two theorems.
\begin{thm*}
\begin{enumerate}
  \item[]
  \item If $\ca$ is a category of saturated objects, then $\mathsf{S}(\ca)$ is an atomic topos.
  \item If in addition $\ca$ has  the joint embedding property, then $\mathsf{S}(\ca)$ is boolean and two valued.
  \item If in addition $\eta_\ca$ is isofull and faithful and surjective on objects, then $\ca$ is categorical in some presentability rank.
\end{enumerate}
\end{thm*} 

\begin{thm*} 
If $\ce$ is an atomic topos, then $\mathsf{pt}(\ce)$ is a \textit{candidate} category of saturated objects.
\end{thm*} 
\end{rem}

%  \begin{thm*}[\ref{thm}]
% Let $\ca \in \text{Acc}_{\omega}$ then the following are equivalent:
% \begin{enumerate}
% \item $\ca$ is a category of saturated objects.
% \item $\mathsf{S}\ca$ is an atomic topos and $\ca$ has a faithful and iso-full functor $f: \ca \to \mathsf{Ind}(C)$ preserving directed colimits.
% \end{enumerate}
%  \end{thm*}

Let us recall (or introduce) the notion of $\omega$-saturated object in an accessible category and the joint embedding property.

\begin{defn}
Let $\ca$ be an accessible category. We say that $s \in \ca$ is $\omega$-saturated if it is injective with respect to maps between finitely presentable objects. That is, given a morphism between finitely presentable objects $f: p \to p'$ and a map $p \to s$, there exists a lift as in the diagram below.
\begin{center}
\begin{tikzcd}
s                     &                       \\
p \arrow[u] \arrow[r] & p' \arrow[lu, dashed]
\end{tikzcd}
\end{center}
\end{defn}

\begin{rem}
In general, when we look at accessible categories from the perspective of model theory, every map in $\ca$ is a monomorphism, and this definition is implicitly adding the hypothesis that every morphism is \textit{injective}.
\end{rem}

\begin{rem}
A very good paper to understand the categorical approach to saturation is \cite{Rsaturated}.
\end{rem}

\begin{defn}
Let $\ca$ be a category. We say that $\ca$ has the joint embedding property if given two objects $A,B$ there exist and object $C$ and two morphisms $A \to C$, $B \to C$.
\end{defn}

\begin{rem}
In \citep{simon}, Henry proves that there are AECs that cannot appear as the category of points of a topos, which means that they cannot be axiomatized in $\text{L}_{\infty, \omega}$. This answers a question initially asked by Rosický at the conference Category Theory 2014 and makes a step towards our understanding of the connection between accessible categories with directed colimits and axiomatizable classes.
The main tool that allows him to achieve this result is called in the paper the \textit{Scott construction}; he proves the Scott topos of $\Set_{\geq \kappa}$\footnote{The category of sets of cardinality at least $\kappa$ and injective functions} is atomic. Even if we developed together the relevant rudiments of the Scott construction, the reason for which this result was true appeared to the author of this paper enigmatic and mysterious. With this motivation in mind we\footnote{The author of this paper.} came to the conclusion that  the Scott topos of $\Set_{\geq \kappa}$ is atomic because of the fact that $\Set_{\geq \kappa}$ appears as a subcategory of saturated objects in $\Set$. 
\end{rem}

\begin{rem}
As a direct corollary of the theorems in this section one gets back the main result of \citep{simon}, but this is not the main accomplishment of this section. Our main contribution is to present a conceptual understanding of \citep{simon} and a neat technical simplification of his proofs. We also improve our poor knowledge of the Scott adjunction, trying to collect and underline its main features. We feel that the Scott adjunction might serve as a tool to have a categorical understanding of Shelah's categoricity conjecture for accessible categories with directed colimits.
\end{rem}

\begin{rem}[What is categoricity and what about the categoricity conjecture?]
Recall that a category of models of some theory is categorical in some cardinality $\kappa$ if it has precisely one model of cardinality $\kappa$. Morley has shown in 1965 that if a category of models is categorical in some cardinal $\kappa$, then it must be categorical in any cardinal above and in any cardinal below up to $\omega_1$ (\cite{chang1990model}). We will be more precise about Morley's result in the section about open problems. When Abstract elementary classes were introduced in the 1970's, Shelah chose Morley's theorem as a sanity check result for his definition. Since then, many approximations of these results has appeared in the literature. The most updated to our knowledge is contained in \cite{vasey2019categoricity}. We recommend the paper also as an introduction to this topic.
\end{rem}

\begin{defn}[(Candidate) categories of ($\omega$-)saturated objects] Let $\ca$ be a category in $\text{Acc}_{\omega}$. We say that $\ca$ is a category of (finitely) saturated objects if there a is topological embedding $j: \ca \to \ck$ in $\text{Acc}_{\omega}$ such that:
\begin{enumerate}
\item $\ck$ is a finitely accessible category.
\item $j\ca \subset \text{Sat}_\omega(\ck)$\footnote{The full subcategory of $\omega$-saturated objects.}.
\item $\ck_{\omega}$ has the amalgamation property\footnote{A category has the amalgamation property is every span can be completed to a square.}.
\end{enumerate}
We say that $\ca$ is a candidate category of (finitely) saturated objects if there exists a functor $j$ that verifies $(1)$-$(3)$.
\end{defn}

\begin{rem}
The notion of \textit{category of saturated objects} axiomatizes the properties of the inclusion $j:\Sat_\omega(\ck) \hookrightarrow \ck$, our motivating example was the inclusion of $\Set_{\geq \kappa} \hookrightarrow \Set_{\geq \omega} \hookrightarrow \Set$. The fact that every object in $\Set_{\geq \kappa}$ is injective with respect to finite sets is essentially the axiom of choice. \citep{Rsaturated} describes a direct connection between saturation and amalgamation property, which was also implied in \citep{Caramelloatomic}.
\end{rem}

In \citep{Caramelloatomic}, Caramello proves - essentially - that the category of points of an atomic topos is a category of saturated objects and she observes that it is countable categorical. This shows that there is a deep connection between categoricity, saturation and atomic topoi. We recall the last notion before going on with the exposition.

\begin{defn}[Characterization of atomic topoi, {\cite[C3.5]{elephant2}}] Let $\cg$ be a Grothendieck topos, then the following are equivalent:
\begin{enumerate}
\item $\cg$ is atomic.
\item $\cg$ is the category of sheaves over an atomic site.
\item The subobject lattice of every object is a complete atomic boolean algebra.
\item Every object can be written as a disjoint union of atoms.
\end{enumerate}
\end{defn}

\begin{thm} \label{thmcategoriesofsaturated objects}
\begin{enumerate}
  \item[]
  \item If $\ca$ is a category of saturated objects, then $\mathsf{S}(\ca)$ is an atomic topos.
  \item If in addition $\ca$ has  the joint embedding property, then $\mathsf{S}(\ca)$ is boolean and two valued.
  \item If in addition $\eta_\ca$ is iso-full, faithful and surjective on objects, then $\ca$ is categorical in some presentability rank.
\end{enumerate}
\end{thm} 
\begin{proof}

\begin{enumerate}
  \item[]
  \item Let $\ca$ be a category of saturated objects $j: \ca \to \ck$. We must show that $\mathsf{S}(\ca)$ is atomic. The idea of the proof is  very simple; we will show that:
    \begin{itemize}
    \item[(a)] $\mathsf{S}j$  presents $\ca$ as  $j^*: \Set^{\ck_\omega} \leftrightarrows \mathsf{S}(\ca):  j_*$;
    \item[(b)] The induced topology on $\ck_\omega$ is atomic.
    \end{itemize}
  (a) follows directly from the definition of topological embedding and \cite[Rem. 2.13]{thcat}. (b) goes identically to \citep[Cor. 4.9]{simon}: note that for any map $k  \to k' \in \ck_{\omega}$, the induced map $j^*yk \to j^*yk'$ is an epimorphism: indeed any map $k \to ja$ with $a \in \ca$ can be extended along $k \to k'$ because $j$ makes $\ca$ a category of saturated objects. So the induced topology on $\ck_{\omega}$ is the atomic topology (every non-empty sieve is a cover). The fact that $\ck_{\omega}$ has the amalgamation property is needed to make the atomic topology a proper topology.
  \item Because $\ca$ has the joint embedding property, its Scott topos is connected. Indeed a topos is connected when the inverse image of the terminal map $t: \mathsf{S}(\ca) \to \Set$ is fully faithful. $t$ appears as the $\mathsf{S}(\tau)$, where $\tau$ is the terminal map $\tau: \ca \to \cdot$. When $\ca$ has the JEP, and thus is connected, $\tau$ is a lax-epi, and $f^*$ is fully faithful by \cite[Prop. 4.6]{thcat}. Then, $\mathsf{S}(\ca)$ is atomic and connected. By \cite[4.2.17]{caramello2018theories} it is boolean two-valued.
  \item This follows from \cite[Prop. 4.13]{thcat} and \cite{Caramelloatomic}. In fact, Caramello has shown that $\mathsf{ptS}(\ca)$ must be countably categorical and the countable object is saturated (by construction). Thus, the unit of the Scott adjunction must reflect the (essential) unicity of such an object.
\end{enumerate}
\end{proof}

\begin{thm} 
If $\ce$ is an atomic topos, then $\mathsf{pt}(\ce)$ is a \textit{candidate} category of saturated objects.
\end{thm}
\begin{proof}
Let $\ce$ be an atomic topos and $i: \ce \to \Set^C$ be a presentation of $\ce$ by an atomic site. It follows from \cite{Caramelloatomic} that $\mathsf{pt}(i)$ presents $\mathsf{pt}(\ce)$ as a candidate category of saturated objects.
\end{proof}

\subsection{Categories of $\kappa$-saturated objects}

Obviously the previous definitions can be generalized to the $\kappa$-case of the Scott adjunction, obtaining analogous results. Let us boldly state them.

\begin{defn}[(Candidate) categories of ($\kappa$-)saturated objects] Let $\ca$ be a category in $\text{Acc}_{\kappa}$. We say that $\ca$ is a category of $\kappa$-saturated objects if there is topological embedding (for the $\mathsf{S}_\kappa$-adjunction) $j: \ca \to \ck$ in $\text{Acc}_{\kappa}$ such that:
\begin{enumerate}
\item $\ck$ is a $\kappa$-accessible category.
\item $j\ca \subset \text{Sat}_\kappa(\ck)$.
\item $\ck_{\kappa}$ has the amalgamation property.
\end{enumerate}
We say that $\ca$ is a candidate category of $\kappa$-saturated objects if there exists a functor $j$ that verifies $(1)$-$(3)$.
\end{defn}

\begin{thm} \label{kappasaturated}
\begin{enumerate}
  \item[]
  \item If $\ca$ is a category of $\kappa$-saturated objects, then $\mathsf{S}_\kappa(\ca)$ is an atomic $\kappa$-topos.
  \item If in addition $\ca$ has  the joint embedding property, then $\mathsf{S}_\kappa(\ca)$ is boolean and two valued.
  \item If in addition $\eta_\ca$\footnote{the unit of the $\kappa$-Scott adjunction.} is iso-full, faithful and surjective on objects, then $\ca$ is categorical in some presentability rank.
\end{enumerate}
\end{thm} 

\begin{thm} 
If $\ce$ is an atomic $\kappa$-topos, then $\mathsf{pt}_\kappa(\ce)$ is a \textit{candidate} category of $\kappa$-saturated objects.
\end{thm}

%\begin{proof}
%This is a direct consequence of \ref{generalize}.
%\end{proof}

\section*{Acknowledgements}
The content of the first section was substantially inspired by some private conversations with Jiří Rosický. I am grateful to Simon Herny for some very constructive discussions on Sec. 3 and 4. I am indebted to Axel Osmond for having read and commented a preliminary draft of this paper.

\bibliography{thebib}
\bibliographystyle{alpha}

\end{document}